\documentclass[
article]{achemso}
\usepackage{amsfonts}
\usepackage{amssymb}
\usepackage{times}
\usepackage{amsmath}
\usepackage{graphicx}
\usepackage{dsfont}
 \usepackage{amsthm}
\newtheorem{lem}{Lemma}
\newtheorem{stw}{Proposition}
\newtheorem{tw}{Theorem}

\newtheorem{wn}{Corollary}
\author{Dariusz Socha}

\email{d.socha@mini.pw.edu.pl}
\affiliation[PW]
{Faculty of Mathematics and Information Sciences,\\
Warsaw University of Technology}
\title[\texttt{achemso} demonstration]
{Existence of the global solutions of an integro-differential  equation in population dynamics}

\begin{document}

\begin{abstract}
We study a nonlinear integro-differential  equation arising in population
dynamics. It has been already proved  by Rybka, Tang and  Waxman
that it has a unique local in time
solution. Here, after deriving appropriate a priori estimates  we
show that the dynamics is global in time. 
\end{abstract}

\section{Introduction}

We study solutions to  an integro-differential
equation (\ref{eq:rownanie}), arising in population dynamics, see \cite{rtw}. 
The specific problem, we have in mind is set up to model the evolution
of an asexual population. Let us describe the basic premise.
The organisms are  born mature and
mutation occurs at the birth. The  selection at time $t>0$ depends
only on a
single  trait  $x$, which is 
assumed to be a real number. 
We study
evolution of a probability density $\phi(x,t)$ of the trait in the population,
taking value  
$x$ at time $t>0$.
This line of theoretical research in biology was initiated by
Kimura \cite{mkimura}. The actual set of equation, we study here, was derived
in \cite{Wax-Peck1} under the premise that the environment changes at a
constant rate $c$. The model equation is as follows,
\begin{equation}\label{eq:rownanie}
\left\{
\begin{array}{l}
\displaystyle{
\frac{\partial\phi}{\partial t}=((1-U)\bar{d}(t)-d(x-ct))\phi+U\bar{d}(t)f\star\phi},\\
\displaystyle{\phi(x,0)=\phi_{0}(x)},\\
\phi(x,t)\geq 0\;\;\mbox{for all}\;\;(x,t)\in(-\infty,\infty)\times(0,\infty),\\
\displaystyle{
\int_{-\infty}^{\infty}\phi(x,t)\,dx=1\;\;\mbox{for all}\;\;t\geq 0 .}
\end{array}\right. 
\end{equation}
Here $U$ is a positive coefficient, which is less than one, $c$ is a
given positive number. The other
ingredients  will be explained below.
This initial value problem has been studied in a series of papers,
\cite{Wax-Peck2}, \cite{fala}, \cite{rtw}, \cite{pakistanczyk}.
In \cite{Wax-Peck2} a discrete time version was studied
and  numerical experiments were performed suggesting 
existence of a traveling wave. This is  a special but very
important form of a solution.
Such a
solution tells us that the organism may follow the changes of the
environment. It also gave insight into
the shape of such moving distribution density. In \cite{broom} the
multi-loci version of (\ref{eq:rownanie}) for
a traveling wave was numerically studied.  Later, 
Becker in her thesis \cite{becker} considered a model of population with
a two coupled traits. 
Since this kind of nonlinear nonlocal
equation seems to be new, there has been  a lot of open theoretical 
questions. In \cite{fala} the authors rigorously established existence of a
traveling wave. These results were subsequently improved by
\cite{pakistanczyk}.  
Well-posedness of system (\ref{eq:rownanie}) was studied in \cite{rtw}, where unique local-in-time solutions were constructed and their
continuous dependence upon the data was shown too. Also the
well-posedness of the system linearized around the traveling wave was
studied. However, two major questions: (1)  global in time existence
of solutions; (2) stability of the traveling wave, were left open. 

In \cite{rtw} the authors missed a crucial a priori estimate for the global
existence. We derive it here, see Proposition 1 and Theorem 1. 
However, we impose
an additional assumption on the initial datum, namely we require that
it has finite first four moments. Our main
accomplishment, which is the global existence stated in Theorem 2 is
based on that estimate. We stress that this fact is not obvious,
since we deal with a nonlinear problem whose growth in
$\phi$ is quadratic.

Let us describe  (\ref{eq:rownanie}) and its
components in details. The equation 
is so normalized that  $x=0$ is the optimal genotypic value. The
mortality rate given by 

$$
d(x) = 1+x^2,
$$
reflects this normalization. Since by assumption the optimal value
changes with the velocity $c$, the mean mortality obtained by
averaging $d$ against $\phi$ is,
\begin{equation}
\bar{d}(t) =\int_{\mathds{R}}d(x-ct)\phi(x,t)dx.
\end{equation}
This suggests that $\phi$ should have at least first two moments. For
any probability distribution $u$ its $n$-th moment is denoted by
$M_n(u)$, 
\begin{equation}\label{momenta}
M_n(u) =\int_{\mathds{R}}x^{n}u(x)\,dx.
\end{equation}

In our equation,  (\ref{eq:rownanie}) $f$ is the normal distribution
with mean $m$ and 
standard deviation $\sigma$, i.e
$$
f(x) =
\frac{1}{\sqrt{2\pi}\sigma}e^{-\frac{(x-m)^{2}}{2{\sigma}^{2}}}. 
$$ 

This choice is biologically
well-motivated, but our analysis holds for any smooth probability
density  with finite sufficiently high moments.

We recall that $f\star\phi$ denotes the convolution of $f$ and $\phi$, i.e.

$$
f\star \phi(x,t)=\int_{\mathds{R}}f(x-y)\phi(y,t)\,dy.
$$

After describing all the ingredients of  (\ref{eq:rownanie}) we
present the plan of our paper. We observe that we have to show enough a
priori estimates to prove global in time existence. We mentioned that
the density $\phi$ must have first two moments. 
In order to guarantee that we found that it is advantageous to work with the
following Banach space,
\[X = \{u \in C^{0}(\mathds{R}):\sum_{j=0}^{2}p_{2j}(u)<\infty \},
\hskip 1em \mbox{where}\;
p_j(u)=||x^{j}u(x)||_{L^{\infty}(\mathds{R})}.\]
In \cite{rtw} it was shown, that for a positive $T>0$ equation
(\ref{eq:rownanie}) system has a 
unique solution $\phi \in C^{1}([0,T),X)$.

It turns out however that, an additional assumption of finiteness of
the fourth moment of initial data guarantees that the solution is
global in time. This is the content of Theorem 2.
The first step in the proof is an observation that the second moment
of the solution defined on interval $(0,T)$ is time integrable,
provided that $T$ is finite. This in turn implies that the fourth
moment is uniformly bounded over $(0,T)$. As a result, we may deduce
that the solution $\phi(\cdot,t)$ is not only bounded in $X$, but also
it forms a Cauchy sequence if $t_k\to T$. This analysis depends in an
essential way on the  constant variation
formula for (\ref{eq:rownanie}).
\section{Estimate of moments}

Experience gained in \cite{rtw} suggests that it is easier to work with
the integral form of 
(\ref{eq:rownanie}). 
The constant variation formula applied to
(\ref{eq:rownanie}) yields,
\begin{equation}\label{g:calkowe}
\begin{split}
\phi(x,t)=&\exp(-\int_{0}^{t}[d(x-cs)-(1-U)\bar{d}(s)]\,ds)\phi_{0}(x)\\
&+\int_{0}^{t}\exp(-\int_{s}^{t}
[d(x-c\tau)-(1-U)\bar{d}(\tau)]d\tau)U\bar{d}(s)f\star\phi(x,s) \,ds. 
\end{split}
\end{equation}
Equation (\ref{g:calkowe}) is indeed nonlinear, for $\bar{d}$ depends
upon $\phi$.
However, for a fixed $\bar{d}$ (\ref{g:calkowe}) becomes
a linear equation for $\phi$.  We will use this property to squeeze out the
first estimate below. 

\begin{stw}
If $\phi$ is  a unique solution to (\ref{eq:rownanie}), defined on
$[0,T)$, then function 
$\bar{d}(s)$ belongs to $L^{1}(0,T)$. 
\end{stw}

\begin{proof}
From now on, we will work with the integral  equation (\ref{g:calkowe}).
Integrating both sides of (\ref{g:calkowe}) with respect to $x$ over $\mathds{R}$, we get
\begin{eqnarray*}
1&=&\int_{\mathds{R}}\phi(x,t)\,dx\\
&=&\int_{\mathds{R}}\Big[\exp\Big(-\int_{0}^{t}[d(x-cs)-(1-U)\bar{d}(s)]\,ds\Big)\phi_{0}(x)\\
&&+\int_{0}^{t}\exp\Big(-\int_{s}^{t}[d(x-c\tau)-(1-U)\bar{d}(\tau)]d\tau\Big)U\bar{d}(s)f\star\phi(x,s)
  \,ds\Big]\,dx.
\end{eqnarray*}
Since  both terms on the right-hand-side (RHS) are non-negative, we
conclude that 

\begin{equation}\label{zdolu}
1\geqslant \exp\Big((1-U)\int_{0}^{t}\bar{d}(s)\,ds\Big)\int_{\mathds{R}}\exp\Big(-\int_{0}^{t}d(x-cs)\,ds\Big)\phi_{0}(x)\,dx.
\end{equation}
In order to extract the desired estimate, we examine the following
expression in detail, 

\[\int_{\mathds{R}}\exp\Big(-\int_{0}^{t}d(x-cs)\,ds\Big)\phi_{0}(x)\,dx=
e^{-t-c^{2}\frac{t^{3}}{3}}\int_{\mathds{R}}e^{-x^{2}t+cxt^2}\phi_{0}(x)\,dx =
e^{-t-c^{2}\frac{t^{3}}{3}}J
.\]
We choose $R=R_{\phi_{0}}$ such that

\begin{equation}\label{jednadruga}
\int_{|x|<R}\phi_{0}(x)=\frac{1}{2}.
\end{equation}
If we take into account that
\[\inf_{\{x\in\mathds{R}:|x|<R\}}\exp\big(-x^{2}t+cxt^2\big)=\exp\big(-R^{2}t-cRt^2\big),\]
then we have the following bound on $J$
\[J\geqslant\int_{|x|<R}e^{-x^{2}t+cxt^2}\phi_{0}(x)\,dx\geqslant e^{-R^{2}t-cRt^2}\int_{|x|<R}\phi_{0}(x)\,dx=\frac{1}{2}e^{-R^{2}t-cRt^2}.\]

The last inequality holds because of (\ref{jednadruga}).
Keeping this is mind, we can deduce that (\ref{zdolu}) implies the
following inequality,
\begin{equation}\label{trzyipol}
C(t,\phi_{0}):=2\exp\big(t+c^{2}\frac{t^{3}}{3}\big)\exp\big(R^{2}t+cRt^2\big)\geqslant
\exp\bigg((1-U)\int_{0}^{t}\bar{d}(s)\,ds\bigg). 
\end{equation}


Since $t \to C(t,\phi_{0})$ is increasing with respect to $t$, $t\leqslant T<\infty$ and $\bar{d}$ is positive we infer that

\[\int_{0}^{T}\bar{d}(s)\,ds\le  \frac{\ln C(T,\phi_{0})}{1-U} <\infty.\]
\end{proof}
Subsequently, we will use simple estimates of  probability density moments.
\begin{lem}\label{momenty0}
For any probability measure $f$, the following inequalities hold 
\[|M_{n}(f)|\leqslant 1+M_{n+2}(f)\;\;\mbox{when  $n$ is even,}\]
\[|M_{n}(f)|\leqslant 1+M_{n+1}(f)\;\;\mbox{when $n$ is odd.}\]
Here, $M_{n}$ is the n-th moment defined by (\ref{momenta}).
\end{lem}
We leave an easy proof to the reader.

\bigskip\noindent
In particular, this lemma implies the following fact.
\begin{wn}\label{wniosek1}
If $f$ is any probability density, then for any $1\leqslant i \leqslant 4$, we have
\[|M_{i}(f)| \leqslant 2+M_{4}(f).\]
\end{wn}
\begin{proof}
Indeed, Lemma 1 gives us
\begin{equation}\label{oszac}
|M_{1}(f)|\leqslant 1+M_{2}(f)\leqslant 2 +M_{4}(f),\;\;|M_{3}(f)|\leqslant 1 + M_{4}(f).
\end{equation}
\end{proof}
Now, we shall show propagation of regularity understood as finiteness
of moments. 
\begin{tw}\label{tw-gl}
Let us suppose that the probability distribution $\phi(\cdot,t)$ is
the solution of (\ref{g:calkowe}) and $\phi_{0}$ has finite first four
moments. Then,\\ 
(a) for all $t\in{[0,T]}$ solution $\phi(\cdot,t)$ has finite first
four moments; moreover, 
\[\sup_{t\in[0,T]}M_{n}(\phi(\cdot,t))<\infty.\hskip 2em n=1,2,3,4.\]
(b) \[\sup_{t\in [0,T)}||\phi(\cdot,t)||_{X}<M(T)<\infty.\]
\end{tw}

\begin{proof}
For the sake of simplicity of notation we will write $M_{n}(t)$ in place of $M_{n}(\phi(\cdot,t))$. We notice that equation (\ref{g:calkowe}), positivity of $\bar{d}$ and estimate (\ref{trzyipol}) imply that
\begin{equation}\label{wstawka}
|M_{n}(t)|\leq C(T,\phi_{0})|M_{n}(0)|
+C(T,\phi_{0})\int_{0}^{t}\bar{d}(s) \int_{\mathds{R}}x^n(f\star\phi)(x,s)\,dxds.
\end{equation}

We want to express the integrand of the RHS of (\ref{wstawka}) in terms of $M_{n}(s)$. For this purpose, we have to transform $x^n$ into a more convenient form.
We recall that Newton binomial formula yields
\[x^{n}=(x-y)^{n}+\sum_{k=1}^{n}{n\choose k}(x-y)^{n-k}y^{k}.\]
If we apply this formula to the integral on the RHS of
(\ref{wstawka}), then we shall see it may be estimated by the
following expression 
\begin{eqnarray*}
&&\bigg|\int_{0}^{t}\bar{d}(s)\int_{\mathds{R}}\int_{\mathds{R}}(x-y)^nf(y)\phi(x-y,s)\,dydxds\bigg|\\
&+&\bigg|\sum_{k=1}^{n}{n \choose
  k}\int_{0}^{t}\bar{d}(s)\int_{\mathds{R}}\int_{\mathds{R}}(x-y)^{n-k}y^kf(y)\phi(x-y,s)\,dydxds\bigg|=:Z_{1}+Z_{2}.
\end{eqnarray*}
If we substitute $z=x-y$ and keep in mind the definition of $f$, then
we notice that  
\[Z_{1}+Z_{2}=\int_{0}^{t}\bar{d}(s)M_{n}(s)\,ds+\sum_{k=1}^{n}{n \choose k}\int_{0}^{t}\bar{d}(s)|M_{k}(f)||M_{n-k}(s)|\,ds\]
Now, we invoke Corollary \ref{wniosek1}, and take $n=4$, then for any
$k=1,2,3$ we have
$$|M_{n-k}(s)|\leq 2+M_{4}(s).
$$
If we write $C_{f}$ for $\max_{k\in\{1,...,n\}}|M_{k}(f)|$, then
(\ref{wstawka}) becomes  
\[M_{4}(s)\leq C(T,\phi_{0})M_{4}(0)+16C_{f}C(T,\phi_{0})\int_{0}^{t}\bar{d}(s)M_{4}(s)\,ds.\]
Since $\bar{d}(s)$ is integrable, Gronwall inequality implies
\[M_{4}(t)\leqslant M_{4}(0)C(T,\phi_{0})
\exp\left({16C_{f}C(T,\phi_{0})\int_{0}^{t}\bar{d}(s)}\,ds\right),
\quad 0\leqslant t \leqslant T.\]
Combining this inequality  with Corollary \ref{wniosek1} yields the claim.

We now show part (b).
By the integral equation and the definition of $p_{0}$, we notice:
\begin{eqnarray}\label{young}
p_{0}(\phi(\cdot,t))&=&||\phi(x,t)||_{L^{\infty}(\mathds{R})}\nonumber\\
&\leq&\Big{|}\Big{|}\exp\Big(\int_{0}^{t}(1-U)\bar{d}(s)]\,ds\Big)\phi_{0}(x)\Big{|}\Big{|}_{L^{\infty}(\mathds{R})}\nonumber\\
&&+\Big{|}\Big{|}\int_{0}^{t}\exp\Big(-\int_{s}^{t}[d(x-c\tau)-(1-U)\bar{d}(\tau)]d\tau\Big)U\bar{d}(s)f\star\phi(x,s)
\,ds\Big{|}\Big{|}_{L^{\infty}(\mathds{R})}\nonumber
\\ &\leqslant&
C(T,\phi_{0})\big{|}\big{|}\phi_{0}\big{|}\big{|}_{L^{\infty}(\mathds{R})}+\frac{C(T,\phi_{0})U}{\sqrt{2\pi}\sigma}\int_{0}^{T}\big{|}\bar{d}(s)\big{|}ds=:P_{0}<\infty.  
\end{eqnarray}
We used in these calculations the following bound
\[\big{|}\big{|}f\star\phi\big{|}\big{|}_{L^{\infty}(\mathds{R})}\leqslant||f||_{L^{\infty}(\mathds{R})}=\frac{1}{\sqrt{2\pi}\sigma}.\]
We can also see that
\begin{eqnarray*}
p_{2}(\phi(\cdot,t))&=&
||x^{2}\phi(x,t)||_{L^{\infty}(\mathds{R})}\\
&\leqslant&
C(T,\phi_{0})\Big{|}\Big{|}x^{2}\phi_{0}\Big{|}\Big{|}_{L^{\infty}(\mathds{R})}
+C(t,\phi_{0})U\Big{|}\Big{|}\int_{0}^{T}x^{2}\bar{d}(s)f\star\phi(x,s) \,ds\Big{|}\Big{|}_{L^{\infty}(\mathds{R})}\\
&\leqslant&
C(T,\phi_{0})
p_{2}(\phi_{0})+C(T,\phi_{0})U\int_{0}^{T}\big{|}\bar{d}(s)\big{|}\big{|}\big{|}x^{2}f\star\phi\big{|}\big{|}_{L^{\infty}(\mathds{R})}
\,ds.
\end{eqnarray*}
In order to estimate the expression above, one can look at
$\big{|}\big{|}x^{2}f\star\phi\big{|}\big{|}_{L^{\infty}(\mathds{R})}.$
Using an elementary inequality $x^{2}\leqslant 2y^{2}+2(x-y)^{2}$, we
get immediately 
\begin{eqnarray*}
\big{|}\big{|}x^{2}f\star\phi(\cdot,t)\big{|}\big{|}_{L^{\infty}(\mathds{R})}
&\leqslant &
2\int_{\mathds{R}}f(x-y)y^{2}\phi(y,t)\,dy+
2\int_{\mathds{R}}\phi(y,t)f(x-y)(x-y)^{2}\,dy\\
&\leqslant&
\frac{\sqrt 2}{\sqrt{\pi}\sigma}\int_{\mathds{R}}y^{2}\phi(y,t)\,dy
+2p_{2}(f)\int_{\mathds{R}}\phi(y,t)\,dy=\frac{\sqrt
  2}{\sqrt{\pi}\sigma}M_{2}(\phi(\cdot,t))+2p_{2}(f) .
\end{eqnarray*} 
Thus,
\begin{equation}\label{p2:p2}
\sup_{t\in[0,T)}p_{2}(\phi(\cdot,t))\leqslant C(T,\phi_{0})p_{2}(\phi_{0})+C(T,\phi_{0})U\int_{0}^{T}\bar{d}(s)\,ds\frac{\sqrt 2}{\sqrt {\pi}\sigma}\sup_{t\in[0,T)}M_{2}(\phi)+2p_{2}(f)=:P_{2}<\infty
\end{equation}
Now, we turn our attention to $p_{4}(\phi(\cdot,t))$.
\begin{eqnarray*}
p_{4}(\phi(\cdot,t))&=&
||x^{4}\phi(x,t)||_{L^{\infty}(\mathds{R})}\\
&\leqslant&
C(T,\phi_{0})
\Big{|}\Big{|}x^{4}\phi_{0}(x)\Big{|}\Big{|}_{L^{\infty}(\mathds{R})}
+C(T,\phi_{0})U\Big{|}\Big{|}
\int_{0}^{T}x^{4}\bar{d}(s)f\star\phi(x,s) \,ds\Big{|}\Big{|}_{L^{\infty}(\mathds{R})}\\
&\leqslant&
p_{4}(\phi_{0})+
\int_{0}^{T}\big{|}\bar{d}(s)\big{|}\big{|}\big{|}x^{4}f\star\phi(x,s)\big{|}\big{|}_{L^{\infty}(\mathds{R})} \,ds.
\end{eqnarray*}

In order to estimate this, we look at 
$\big{|}\big{|}x^{4}f\star\phi(x,s)\big{|}\big{|}_{L^{\infty}(\mathds{R})}$.
We notice that
for all real $x,y$ the following inequality holds:
\[x^{4}\leqslant 5y^4 + 5(x-y)^4+6y^2(x-y)^2.\]
Using this inequality, we continue our calculations,

\begin{eqnarray*}
\big{|}\big{|}x^{4}f\star\phi\big{|}\big{|}_{L^{\infty}(\mathds{R})}
&\leqslant&
 5\int_{\mathds{R}}f(x-y)y^{4}\phi(y,s)\,dy\\
&&+6\int_{\mathds{R}}\phi(y,s)f(x-y)y^{2}(x-y)^{2}\,dy
+5\int_{\mathds{R}}\phi(y,s)f(x-y)(x-y)^{4}\,dy\\
&\leqslant&
\frac{5}{\sqrt{2\pi}\sigma}M_{4}(\phi)+6M_{2}(\phi)p_{2}(f)+5p_{4}(f).
\end{eqnarray*}




Therefore,

\begin{equation}\label{p4:p4}
\sup_{t\in_[0,T)}p_{4}(\phi(\cdot,t))\leqslant 
C(T,\phi_{0}) p_{4}(\phi_{0})+
C(T,\phi_{0})UC_{4}\int_{0}^{T}\big{|}\bar{d}(s)\big{|}\,ds=:P_{4}.
\end{equation}
Now, we can take for $M =\max\{P_0, P_2, P_4\}$.
\end{proof}
\section{Existence of global solutions}
In this section we show our main result, the global in time solution
to (\ref{eq:rownanie}). It is based upon the estimates established in
Section 2. They use in an essential  way the additional regularity of
the initial 
condition, i.e.  the finite fourth moment. These bounds depend upon time,
but for each time constant $t$ they are finite. 
\begin{tw}
Let us assume that $\phi_{0}\in X$ and $M_{4}(\phi_{0})<\infty$, then the unique solution to (\ref{eq:rownanie}) exists for all $t\in[0,\infty)$, moreover $\phi\in C([0,\infty),X)$. 
\end{tw}

\begin{proof}
We have already constructed a unique solution $\phi\in C([0,T),X)$. If
  $T<\infty$, then we will 
  show, that $\lim_{t\to T}\phi(t)$ exists in $X$. Hence, due to Theorem
  \ref{tw-gl} $\phi(T)$ automatically has a finite fourth moment,
  $M_{4}\phi(T)$. Subsequently, $\phi$ may be extended to
  $[T,T+\epsilon)$ for a positive $\epsilon$. Hence our claim will follow. 

Let us suppose that $\{t_{m}\}$ is any sequence converging to $T$. We shall
see, that  
$\{\phi(t_n)\}$ is a Cauchy sequence in the Banach space $X$. In
order to achieve this  goal we will use the constant variation formula.
For a given $\epsilon>0$ we shall estimate the norm of the difference
$\phi_{n}-\phi_{m}$, 
\begin{equation}
\big{|}\big{|}\phi_{n}-\phi_{m}\big{|}\big{|}_{X}=\big{|}\big{|}\phi(x,t_{n})-\phi(x,t_{m})\big{|}\big{|}_{L^{\infty}(\mathds{R})}+\big{|}\big{|}x^{2}\big(\phi(x,t_{n})-\phi(x,t_{m})\big)\big{|}\big{|}_{L^{\infty}(\mathds{R})}+\\
+\big{|}\big{|}x^{4}\big(\phi(x,t_{n})-\phi(x,t_{m})\big)\big{|}\big{|}_{L^{\infty}(\mathds{R})}.
\nonumber
\end{equation}
In order to simplify our subsequent calculations, we introduce the
following  notation 
\begin{equation}\label{gje}
g(x,s,t_{n})=\int_{s}^{t_{n}}[d(x-c\tau)-(1-U)\bar{d}(\tau)]d\tau.
\end{equation}

The integral equation (\ref{g:calkowe}) implies that for
$j\in\{0,2,4\}$, we have 
\[|x^{j}(\phi_{n}-\phi_{m})(x,\cdot)|\leqslant I_{j1}(x)+I_{j2}(x),\]
where
\[I_{j1}(x)=
\bigg{|}e^{-g(x,0,t_n)}\big(1-e^{-g(x,t_n,t_m)}\big)x^{j}\phi_{0}(x)\bigg{|},\]
\[I_{j2}(x)=
\bigg{|}\int_{0}^{t_{n}}e^{-g(x,s,t_n)}\big(1-e^{-g(x,t_n,t_m)}\big)U\bar{d}(s)x^{j}f\star\phi(x,s)+ \int_{t_{n}}^{t_{m}}e^{-g(x,s,t_m)}U\bar{d}(s)x^{j}f\star\phi(s,x)\bigg{|} 
\,ds.\]
In the following calculations, without loss of generality, we may
assume that $t_{n}<t_{m}$.
First, we shall take care of $I_{j1}(x)$, $j\in\{0,2,4\}$. 
Let us notice first that $\sup_{x\in\mathds{R}}
I_{j1}(x)=\max\{\sup_{|x|\le R}I_{j1}(x),\sup_{|x|>R}I_{j1}(x)\}.$ 
Subsequently, positivity
of $d$, (\ref{gje}) and (\ref{trzyipol}) imply that 
$$
\sup_{|x|\le R}I_{j1}(x) \le C(\phi_0,T)\sup_{|x|\le R}
\big|
\big(1-e^{-\int_{t_{n}}^{t_{m}}d(x-cs)\,ds}\big)\phi_{0}x^j
\big|.
$$
If $R^2\ge 1+2c^2T^2$, then we have
\[d(x-ct)\leqslant 1+2R^2+2c^2T^2<3R^2.\]
Once we fixed $R$ we can see that
\[\sup_{|x|\le R}I_{j1}(x)\leqslant
C(\phi_0,T)(1-e^{-3R^2(t_{m}-t_{n})})||\phi_{0}(x)||_{X}\leqslant
\epsilon\;\;\;\mbox{when $t_{m}-t_{n}$ is small enough.}\]

Now, we turn our attention to the estimate of $\sup_{|x|>R}I_{j1}(x)$.
We may assume that $R$ is so large, that for our $\epsilon$ the
following bound holds
\[C(\phi_0,T)e^{-\int_{0}^{t_{m}}d(x-cs)\,ds}\leqslant
\frac{\epsilon}{||\phi_{0}(x)||_{X}}\;\;\mbox{for}\;\;|x|>R.\] 
Therefore, again due to positivity of $d$, we have
\[\sup_{|x|>R}I_{j1}(x)\leqslant C(\phi_0,T)
\big|\frac{\epsilon}{||\phi_{0}(x)||_{X}}\big(1-e^{-\int_{t_{n}}^{t_{m}}d(x-cs)\,ds}\big)\phi_{0}x^j\big|
\le \epsilon
.\]

The estimates of $I_{j2}(x)$ are performed in similar fashion, but they
have to be slightly more subtle. The problem is that $s\in[0,t_{n}]$
appearing under the integral sign, need not be close to $T$.

Thus, $I_{j2}(x)$, $j=0,2,4$ take the following form
\[I_{j2}(x)\le\bigg{|}\int_{0}^{t_{n}}
\big(
e^{-g(x,s,t_{n})} -
  e^{-g(x,s,t_{m})}\big)U\bar{d}(s)x^{j}f\star\phi(x,s) \,ds
\bigg{|}
+\bigg{|}
\int_{t_{n}}^{t_{m}}e^{-g(x,s,t_m)}U\bar{d}(s)x^{j}f\star\phi(x,s) \,ds\bigg{|}.\]
The second term is easy to estimate. For any $j\in\{0,2,4\}$
integrability of $\bar{d}$ and (\ref{gje}) imply that 
\[\bigg{|}\int_{t_{n}}^{t_{m}}e^{-g(x,s,t_m)}U\bar{d}(s)x^{j}f\star\phi(x,s)
\,ds\bigg{|}\leqslant UP_{j}\int_{t_{n}}^{t_{m}}e^{||d||_{L^{1}}}\to
0,\;\;\mbox{as}\;\;t_n,t_m\to T,\]
where $P_j$,  $j=0,2,4$, were defined in (\ref{young}), (\ref{p2:p2}),
(\ref{p4:p4}). 

In order to bound the first integral, we split the interval $[0,t_{n}]$ into two
$[0,t_{n}-\delta]$ and $[t_{n}-\delta,t_m]$, where $\delta>0$ shall be
selected later. 

We notice that (\ref{trzyipol}),  Theorem \ref{tw-gl} and positivity
of $d$ yield the 
following estimate, 
\[\Big|\int_{t_{n}-\delta}^{t_{n}}e^{-g(x,s,t_{n})}
(1-e^{-(g(x,t_{n},t_m))})U\bar{d}(s)x^{j}f\star\phi(x,s)
\,ds\Big|
\leqslant C(\phi_0,T)U
\Big|\int_{t_{n}-\delta}^{t_{n}}\bar{d}(s)x^{j}f\star\phi(x,s)
\,ds\Big|\leqslant\] 
\[\leqslant CP_{j}\int_{t_{n}-\delta}^{t_{n}}\bar{d}(s)\,ds<\epsilon.\]
The last inequality is a result of integrability of $\bar{d}$ implying
existence of sufficiently small $\delta$ with the desired property.

Now, we have to estimate the supremum over the real number of the
integral over $[0,t_{n}-\delta]$. First we notice that

$$\Big|\int_0^{t_{n}-\delta}e^{-g(x,s,t_{n})}
(1-e^{-(g(x,t_{n},t_m))})U\bar{d}(s)x^{j}f\star\phi(x,s)
\,ds \Big|\leq$$
$$ P_j C(\phi_0,T)\sup_{x\in\mathds{R}}
\Big(
\int_0^{t_{n}-\delta} e^{-\int_s^{t_n} d(x-c\tau)\,d\tau}(1- e^{-\int_{t_n}^{t_m} d(x-c\tau)\,d\tau}) 
\bar{d}(s)\,ds$$
$$+\int_0^{t_{n}-\delta} e^{-\int_{t_n}^{t_m}  d(x-c\tau)\,d\tau}
(1- e^{\int_{t_n}^{t_m}\bar{d}(\tau )\,d\tau})\bar{d}(s)\,ds
\Big)\leq$$
$$P_j C(\phi_0,T)\left(\sup_{x\in\mathds{R}}
\int_0^{t_{n}-\delta}  e^{-\int_s^{t_n} d(x-c\tau)\,d\tau}
(1- e^{-\int_{t_n}^{t_m} d(x-c\tau)\,d\tau}) 
\bar{d}(s)\,ds
+\int_0^{t_{n}-\delta}
(1- e^{\int_{t_n}^{t_m}\bar{d}(\tau )\,d\tau})\bar{d}(s)\,ds
\right).$$

We have already noticed that the second integral can be made smaller
than $\epsilon$ for large $n$ and $m.$

In order to estimate the first integral on the RHS above we use again that for
any positive function $u$, $\sup_{x\in\mathds{R}} u =\max\{
\sup_{|x| >{R}} u,\sup_{|x| \le R} u\}.$ For a given $\epsilon>0$ and
$\delta$ fixed above we may possibly take larger $R>0$ so that,
\[
3R^2 \ge d(x-c\tau)\geqslant\frac{1}{2}R^{2}
\;\;\mbox{and}\;\;
e^{-\frac{1}{2}R^2\delta}<\frac{\epsilon}{P_{j}||\bar{d}||_{L^1}}.\]
With this choice of $R$ for $|x| \le R$ we see that for sufficiently
large $m, n$ we have,
$$
1-e^{-\int_{t_{n}}^{t_{m}}d(x-c\tau)d\tau} \le
1-e^{-\int_{t_{n}}^{t_{m}}3R^2\,ds} <\epsilon.
$$
We conclude that
$$
\sup_{|x| \le R}
\int_0^{t_{n}-\delta} e^{-\int_s^{t_n} d(x-c\tau)\,d\tau}
(1- e^{-\int_{t_n}^{t_m} d(x-c\tau)\,d\tau}) \bar{d}(s)
\,ds
$$
can be made as small as we wish. 
Moreover, 
we have,
\[\sup_{|x| > R}
\int_0^{t_{n}-\delta} e^{-\int_s^{t_n} d(x-c\tau)\,d\tau}(1- e^{-\int_{t_n}^{t_m} d(x-c\tau)\,d\tau})\bar{d}(s)
\,ds
\leqslant C\int_{0}^{t_{n}-\delta}e^{-\frac{1}{2}R^2\delta}\bar{d}(s)
\,ds<\epsilon.
\]
We then conclude that
\[\sup_{x\in\mathds{R}}I_{j2}(x) \leqslant3\epsilon.\]
Subsequently, $\{\phi(t_n)\}$ is a Cauchy sequence, hence the limit
$\lim_{t\to T}\phi(t)$ exists in $X$ and we may invoke the existence
result in \cite{rtw} to continue the solution onto $[T,
  T+\epsilon)$. Our claim follows.
\end{proof}




\end{document}